\documentclass{amsart}
\usepackage{graphicx, color}
\usepackage{amscd}
\usepackage{amsmath,empheq}
\usepackage{amsfonts}
\usepackage{amssymb}
\usepackage{mathrsfs}
\usepackage[all]{xy}
\newtheorem{theorem}{Theorem}[section]
\newtheorem{example}{Example}[section]
\newtheorem{lemma}[theorem]{Lemma}
\newtheorem{proposition}[theorem]{Proposition}
\newtheorem{remark}{Remark}[section]

\newtheorem{claim}{Claim}[section]

\newenvironment{proof-sketch}{\noindent{\bf Sketch of Proof}\hspace*{1em}}{\qed\bigskip}
 
\everymath{\displaystyle} 
\newcommand{\RR}{\mathbb R}

\renewcommand{\leq}{\leqslant}

\renewcommand{\geq}{\geqslant}

\begin{document}
\title[Positive solutions for nonlinear Neumann problems]{Positive solutions for nonlinear Neumann problems with singular terms and convection}
\author[N.S. Papageorgiou]{Nikolaos S. Papageorgiou}
\address[N.S. Papageorgiou]{ Department of Mathematics,
National Technical University,
				Zografou Campus, 15780 Athens, Greece \& Institute of Mathematics, Physics and Mechanics, Jadranska 19, 1000 Ljubljana, Slovenia}
\email{\tt npapg@math.ntua.gr}
\author[V.D. R\u{a}dulescu]{Vicen\c{t}iu D. R\u{a}dulescu}
\address[V.D. R\u{a}dulescu]{Institute of Mathematics, Physics and Mechanics, 1000 Ljubljana, Slovenia \& Faculty of Applied Mathematics, AGH University of Science and Technology, 30-059 Krak\'ow, Poland \& Institute of Mathematics ``Simion Stoilow" of the Romanian Academy, 
          014700 Bucharest, Romania}
\email{\tt vicentiu.radulescu@imfm.si}
\author[D.D. Repov\v{s}]{Du\v{s}an D. Repov\v{s}}
\address[D.D. Repov\v{s}]{Faculty of Education and Faculty of Mathematics and Physics, University of Ljubljana, 1000 Ljubljana, Slovenia \& Institute of Mathematics, Physics and Mechanics, 1000 Ljubljana, Slovenia}
\email{\tt dusan.repovs@guest.arnes.si}
\begin{abstract}
We consider a nonlinear Neumann problem driven by the $p$-Laplacian. In the reaction term we have the competing effects of a singular and a convection term. Using a topological approach based on the Leray-Schauder alternative principle together with suitable truncation and comparison techniques, we show that the problem has positive smooth solutions.
\end{abstract}   
\maketitle

\section{Introduction}

Let $\Omega\subseteq\RR^N$ be a bounded domain with a $C^2$-boundary $\partial\Omega$. In this paper, we study the following nonlinear Neumann problem with singular and convection terms
\begin{equation}\label{eq1}
	\left\{\begin{array}{l}
		-\Delta_pu(z)+\xi(z)u(z)^{p-1}=u(z)^{-\gamma}+f(z,u(z),Du(z))\ \mbox{in}\ \Omega,\\
\displaystyle		\frac{\partial u}{\partial n}=0\ \mbox{on}\ \partial\Omega,\ u>0,\ 1<p<\infty,\ 0<\gamma<1.
	\end{array}\right\}
\end{equation}

In this problem, $\Delta_p$ denotes the $p$-Laplacian differential operator defined by
$$\Delta_pu={\rm div}\,(|Du|^{p-2}Du)\ \mbox{for all}\ u\in W^{1,p}(\Omega),\ 1<p<\infty.$$

In the reaction
term
(the right-hand side) of the problem, we have the competing effects of the singular term $u^{-\gamma}$ and the convection term $f(z,x,y)$ (that is, the perturbation $f$ depends also on the gradient $Du$). The function $f(z,x,y)$ is Carath\'eodory (that is, for all $(x,y)\in\RR\times\RR^N$ the mapping $z\mapsto f(z,x,y)$ is measurable, and for almost all $z\in\Omega$ the mapping $(x,y)\mapsto f(z,x,y)$ is continuous).

The key feature of this paper is that we do not impose any global growth conditions on the function $f(z,\cdot,y)$. Instead, we assume that $f(z,\cdot,y)$ exhibits a kind of oscillatory behavior near zero. In this way we can employ truncation techniques and avoid any growth condition at $+\infty$. In the boundary condition, $\frac{\partial u}{\partial n}$ denotes the normal derivative of $u$, with $n(\cdot)$ being the outward unit normal on $\partial\Omega$.

The presence of the gradient $Du$ in the perturbation $f$, excludes from consideration a variational approach to dealing with (\ref{eq1}). Instead, our main tool is topological and is based on the fixed point theory, in particular, on the Leray-Schauder principle (see Section 2).

Equations with singular terms and equations with convection terms have been investigated separately, primarily in the context of Dirichlet problems. For singular problems, we mention the works of Giacomoni, Schindler \& Takac \cite{8}, Hirano, Saccon \& Shioji \cite{2}, Papageorgiou \& R\u{a}dulescu \cite{17}, Papageorgiou, R\u{a}dulescu \& Repov\v{s} \cite{21,prrbook}, Papageorgiou \& Smyrlis \cite{22, 23}, Perera \& Zhang \cite{24}, and Su, Wu \& Long \cite{27}. For problems with convection, we mention the works of de Figueiredo, Girardi \& Matzeu \cite{2}, Gasinski \& Papageorgiou \cite{6}, Girardi \& Matzeu \cite{9}, Huy, Quan \& Khanh \cite{14}, Papageorgiou, R\u{a}dulescu \& Repov\v{s} \cite{20}, and Ruiz \cite{26}. Of the aformentioned works, only Gasinski \& Papageorgiou \cite{6} and Papageorgiou, R\u{a}dulescu \& Repov\v{s} \cite{20} go outside the Dirichlet framework and deal with Neumann problems. A good treatment of semilinear parametric elliptic equations with both singular and convection terms and Dirichlet boundary condition can be found in Ghergu \& R\u{a}dulescu \cite[Chapter 9]{7}.

\section{Mathematical background and hypotheses}

As we have already mentioned, our method of proof is topological and is based on the fixed point theory, in particular, on the Leray-Schauder alternative principle.

Let $V,\,Y$ be Banach spaces and $g:V\rightarrow Y$ a map. We say that $g(\cdot)$ is ``compact" if $g(\cdot)$ is continuous and maps bounded sets of $V$ into relatively compact subsets of $Y$.

We now recall the Leray-Schauder alternative principle (see, for example, Gasinski \& Papageorgiou \cite[p. 827]{3}  or Granas \& Dugundji \cite[p. 124]{10}).
\begin{theorem}\label{th1}
	If $X$ is a Banach space and $g:X\rightarrow X$ is compact, then one of the following two statements is true:
	\begin{itemize}
		\item[(a)] $g(\cdot)$ has a fixed point;
		\item[(b)] the set $K(g)=\{u\in X:u=tg(u),\ 0<t<1\}$ is unbounded.
	\end{itemize}
\end{theorem}

In what follows, we denote by $\left\langle \cdot, \cdot\right\rangle$  the duality brackets for the pair $(W^{1,p}(\Omega)^*,W^{1,p}(\Omega))$ and by $||\cdot||$ the norm on $W^{1,p}(\Omega)$. Hence
$$||u||=\left(||u||^p_p+||Du||^p_p\right)^{1/p}\ \mbox{for all}\ u\in W^{1,p}(\Omega).$$

In the analysis of problem (\ref{eq1}), we will make use of the Banach space $C^1(\overline{\Omega})$. This is an ordered Banach space with positive (order) cone
$$C_+=\{u\in C^1(\overline{\Omega}):u(z)\geq 0\ \mbox{for all}\ z\in\overline{\Omega}\}.$$

This cone has a nonempty interior which is given by
$$D_+=\{u\in C_+:u(z)>0\ \mbox{for all}\ z\in\overline{\Omega}\}.$$

In fact, $D_+$ is also the interior of $C_+$ when the latter is furnished with the relative $C(\overline{\Omega})$-norm topology.

Let $A:W^{1,p}(\Omega)\rightarrow W^{1,p}(\Omega)^*$ be the nonlinear operator defined by
$$\left\langle A(u),h\right\rangle=\int_{\Omega}|Du|^{p-2}(Du,Dh)_{\RR^N}dz\ \mbox{for all}\ u,h\in W^{1,p}(\Omega).$$

The next proposition summarizes the main properties of this operator (see Motreanu, Motreanu \& Papageorgiou \cite[p. 40]{16}).
\begin{proposition}\label{prop2}
	The operator $A:W^{1,p}(\Omega)\rightarrow W^{1,p}(\Omega)^*$ is bounded (that is, $A$ maps bounded sets to bounded sets), continuous, monotone (hence also maximal monotone) and of type $(S)_+$, that is,
	$$u_n\stackrel{w}{\rightarrow}u\ \mbox{in}\ W^{1,p}(\Omega)\ \mbox{and}\ \limsup\limits_{n\rightarrow\infty}\left\langle A(u_n),u_n-u\right\rangle\leq 0\Rightarrow u_n\rightarrow u\ \mbox{in}\ W^{1,p}(\Omega).$$
\end{proposition}

For the potential function $\xi(\cdot)$, we assume the following:

\smallskip
$H(\xi):$ $\xi\in L^{\infty}(\Omega),\ \xi(z)\geq 0$ for almost all $z\in\Omega$, $\xi\not\equiv 0$.

\smallskip
The following lemma will be helpful in producing estimates in our proofs.
\begin{lemma}\label{lem3}
	If hypothesis $H(\xi)$ holds, then there exists $c_1>0$ such that
	$$\vartheta(u)=||Du||^p_p+\int_{\Omega}\xi(z)|u|^pdz\geq c_1||u||^p\ \mbox{for all}\ u\in W^{1,p}(\Omega).$$
\end{lemma}
\begin{proof}
	Evidently, $\vartheta\geq 0$. Suppose that the lemma is not true. Exploiting the $p$-homogeneity of $\vartheta(\cdot)$ we can find $\{u_n\}_{n\geq 1}\subseteq W^{1,p}(\Omega)$ such that
	\begin{equation}\label{eq2}
		||u_n||=1\ \mbox{and}\ \vartheta(u_n)\leq\frac{1}{n}\ \mbox{for all}\ n\in{\mathbb N}.
	\end{equation}
	
	We may assume that
	\begin{equation}\label{eq3}
		u_n\stackrel{w}{\rightarrow}u\ \mbox{in}\ W^{1,p}(\Omega)\ \mbox{and}\ u_n\rightarrow u\ \mbox{in}\ L^p(\Omega)\ \mbox{as}\ n\rightarrow\infty.
	\end{equation}
	
	Clearly, $\vartheta(\cdot)$ is sequentially weakly lower semicontinuous. So, it follows from (\ref{eq2}) and (\ref{eq3}) that
		\begin{eqnarray}\label{eq4}
			&&\vartheta(u)=0,\\
			&\Rightarrow&u\equiv\eta\in\RR.\nonumber
		\end{eqnarray}
		
		If $\eta=0$, then $u_n\rightarrow 0$ in $W^{1,p}(\Omega)$, which contradicts (\ref{eq2}). So $\eta\neq 0$. Then
		$$0=|\eta|^p\int_{\Omega}\xi(z)dz>0\ (\mbox{see \cite{4} and hypothesis}\ H(\xi)),$$
		which is a contradiction. The proof 
of Lemma~\ref{lem3}
is now complete.
\end{proof}

Let $x\in\RR$ and $x^{\pm}=\max\{\pm x,0\}$. Then for all $u\in W^{1,p}(\Omega)$, we set $u^{\pm}(\cdot)=u(\cdot)^{\pm}$. We have
$$u^{\pm}\in W^{1,p}(\Omega),\ u=u^+-u^-,\ |u|=u^++u^-.$$

We denote by $|\cdot|_N$  the Lebesgue measure on $\RR^N$. Given $u,v\in W^{1,p}(\Omega)$ with $u\leq v$,  define
$$[u,v]=\{y\in W^{1,p}(\Omega):u(z)\leq y(z)\leq v(z)\ \mbox{for almost all}\ z\in\Omega\}.$$

Also, we denote by ${\rm int}_{C^1(\overline{\Omega})}[u,v]$ the interior  of $[u,v]\cap C^1(\overline{\Omega})$ in the $C^1(\overline{\Omega})$-norm topology.
Finally, if $1<p<\infty$, we denote by $p'>1$  the conjugate exponent of $p>1$, that is, $\frac{1}{p}+\frac{1}{p'}=1$.

\smallskip
Now we can introduce our hypotheses on $f(z,x,y)$:

\smallskip
$H(f):$ $f:\Omega\times\RR\times\RR^N\rightarrow\RR$ is a
Carath\'eodory function such that $f(z,0,y)=0$ for almost all $z\in\Omega$ and all $y\in\RR^N$, and
the following properties hold:
\begin{itemize}
	\item[(i)] there exists a function $w\in W^{1,p}(\Omega)\cap C(\overline{\Omega})$ such that $\Delta_p w\in L^{p'}(\Omega)$ and
	\begin{eqnarray*}
		&&0<\hat{c}\leq w(z)\ \mbox{for all}\ z\in\overline{\Omega},-\Delta_pw(z)+\xi(z)w(z)^{p-1}\geq 0\ \mbox{for almost all}\ z\in\Omega,\\
		&&w(z)^{-\gamma}+f(z,w(z),y)\leq-c^*<0\ \mbox{for almost all}\ z\in\Omega\ \mbox{and all}\ y\in\RR^N,
	\end{eqnarray*}
	and if $\rho=||w||_{\infty}$, there exists $\hat{a}_{\rho}\in L^{\infty}(\Omega)$ such that
	$$|f(z,x,y)|\leq\hat{a}_{\rho}(z)[1+|y|^{p-1}]$$
	for almost all $z\in\Omega$, all $0\leq x\leq\rho,$ and all $y\in\RR^N$;
	\item[(ii)] there exists $\delta_0>0$ such that $f(z,x,y)\geq\tilde{c}_{\delta}>0$ for almost all $z\in\Omega$ and all $0<\delta\leq x\leq\delta_0$,  $y\in\RR^N$;
	\item[(iii)] there exists $\hat{\xi}_{\rho}>0$ such that for almost all $z\in\Omega$ and all $y\in\RR^N$ the mapping
	$$x\mapsto f(z,x,y)+\hat{\xi}_{\rho}x^{p-1}$$
	is nondecreasing on $[0,\rho]$, and for almost all $z\in\Omega$, all $0\leq x\leq \rho$, $y\in\RR^N$, and  $t\in(0,1)$, we have
	\begin{equation}\label{eq5}
		f(z,\frac{1}{t}x,y)\leq\frac{1}{t^{p-1}}f(z,x,y).
	\end{equation}
\end{itemize}

\begin{remark}
	Our aim is to produce positive solutions and all the above hypotheses concern the positive semi-axis $\RR_+=\left[0,+\infty\right)$. So, for simplicity, we may assume that
	\begin{equation}\label{eq6}
		f(z,x,y)=0\ \mbox{for almost all}\ z\in\Omega\ \mbox{and all}\ x\leq 0,\  y\in\RR^N.
	\end{equation}
	
	Hypothesis $H(f)(i)$ is satisfied if, for example, there exists $\eta\in(0,+\infty)$ such that $\eta^{-\gamma}+f(z,\eta,y)\leq-c^*<0$ for almost all $z\in\Omega$ and all $y\in\RR^N$. Hypotheses $H(f)(i),(ii)$ together determine the oscillatory behavior of $f(z,\cdot,y)$ near $0^+$. Hypothesis $H(f)(iii)$ is satisfied if we set $f(z,x,y)=0$ for almost all $z\in\Omega$ and all $x\geq w(z)$,  $y\in\RR^N$ and require that the function $x\mapsto\frac{f(z,x,y)}{x^{p-1}}$ is nonincreasing on $\left(0,w(z)\right]$ for almost all $z\in\Omega$ and all $y\in\RR^N$.
\end{remark}
\begin{example}\label{ex}
	The following function satisfies hypotheses $H(f)$. For the sake of simplicity we drop the $z$-dependence and require $\xi(z)\geq c_0^*>0$ for almost all $z\in\Omega$:
	$$f(z,y)=(z^{p-1}-cz^{\tau-1})(1+|y|^{p-1})$$
	for all $0\leq x\leq 1$, $y\in\RR^N$, with $1<p<\tau<\infty$, and $c<2^{\frac{1}{\tau-1}}$.
\end{example}

Finally, we mention that $0<\gamma<1$.
When the differential operator is singular (that is, $1<p<2$), 
 we require that $\gamma\leq (p-1)^2$, which is equivalent to saying that $1+\frac{\gamma}{p-1}\leq p$.

\section{A singular problem}

In this section we deal with the following purely singular Neumann problem:
\begin{equation}\label{eq7}
\left\{\begin{array}{l}
	-\Delta_pu(z)+\xi(z)u(z)^{p-1}=u(z)^{-\gamma}\ \mbox{in}\ \Omega,\\
\displaystyle		\frac{\partial u}{\partial n}=0\ \mbox{on}\ \partial\Omega,\ u>0.
\end{array}\right\}
\end{equation}

Recall that $\vartheta:W^{1,p}(\Omega)\rightarrow\RR$ is the $C^1$-functional defined by
$$\vartheta(u)=||Du||^p_p+\int_{\Omega}\xi(z)|u|^pdz\ \mbox{for all}\ u\in W^{1,p}(\Omega).$$
\begin{proposition}\label{prop4}
	If hypotheses $H(\xi)$ hold, then problem (\ref{eq7}) has a unique positive solution $\bar{u}\in D_+$.
\end{proposition}
\begin{proof}
	Let $\epsilon>0$ and consider the $C^1$-functional $\psi_{\epsilon}:W^{1,p}(\Omega)\rightarrow\RR$ defined by
	$$\psi_{\epsilon}(u)=\frac{1}{p}\vartheta(u)-\frac{1}{1-\gamma}\int_{\Omega}[(u^+)^p+\epsilon]^{\frac{1-\gamma}{p}}dz\ \mbox{for all}\ u\in W^{1,p}(\Omega).$$
	
	Using Lemma \ref{lem3}, we obtain
	\begin{eqnarray*}
		&&\psi_{\epsilon}(u)\geq\frac{c_1}{p}||u||^p-\frac{1}{1-\gamma}\int_{\Omega}(u^+)^{1-\gamma}dz-c_2\ \mbox{for some}\ c_2>0\\
		&\Rightarrow&\psi_{\epsilon}(\cdot)\ \mbox{is coercive}.
	\end{eqnarray*}
	
	Using the Sobolev embedding theorem, we can easily see that the functional $\psi_{\epsilon}(\cdot)$ is sequentially weakly lower semicontinuous. So, by the Weierstrass-Tonelli theorem, we can find $u_{\epsilon}\in W^{1,p}(\Omega)$ such that
	\begin{equation}\label{eq8}
		\psi_{\epsilon}(u_{\epsilon})=\inf\left\{\psi_{\epsilon}(u):u\in W^{1,p}(\Omega)\right\}.
	\end{equation}
	
	Let $s\in(0,1)$. Then
	\begin{eqnarray}\label{eq9} \psi_{\epsilon}(s)&<&\left(\frac{s^p}{p}||\xi||_{\infty}-\frac{s^{1-\gamma}}{1-\gamma}\right)|\Omega|_N\ \mbox{(see hypothesis $H(\xi)$)}\nonumber\\
		&<&\left(\frac{s^p}{p}||\xi||_{\infty}+\frac{1}{1-\gamma}(\epsilon^{\frac{1-
\gamma}{p}}-s^{1-\gamma})\right)|\Omega|_N.
	\end{eqnarray}
	
	If $s>2\epsilon^{1/p}$, then
	\begin{eqnarray}\label{eq10}
		&&\frac{s^p}{p}||\xi||_{\infty}+\frac{1}{1-\gamma}(\epsilon^{\frac{1-\gamma}{p}}-s^{1-\gamma})\nonumber\\
		&<&\frac{s^p}{p}||\xi||_p-\frac{s^{1-\gamma}}{1-\gamma}\left(1-\frac{1}{2^{1-\gamma}}\right)=\tau(s).
	\end{eqnarray}
	
	Recall that $s\in(0,1)$ and note that $0<1-\gamma<1<p$. So, we can find small
	enough  $\hat{s}\in(0,1)$  such that
	\begin{equation}\label{eq11}
		\tau(\hat{s})<0.
	\end{equation}
	
	Then (\ref{eq9}), (\ref{eq10}), (\ref{eq11}) imply that for small 
	enough $\epsilon\in\left(0,\left(\frac{\hat{s}}{2}\right)^p\right)$, we have
	\begin{eqnarray*}		&&\psi_{\epsilon}(\hat{s})<\psi_{\epsilon}(0)=-\frac{1}{1-\gamma}\epsilon^{\frac{1-\gamma}{p}}|\Omega|_N,\\
		&\Rightarrow&\psi_{\epsilon}(u_{\epsilon})<\psi_{\epsilon}(0)\ \mbox{(see (\ref{eq8}))},\\
		&\Rightarrow&u_{\epsilon}\neq 0.
	\end{eqnarray*}
	
	From (\ref{eq8}) we have
	\begin{eqnarray}\label{eq12}
		&&\psi'_{\epsilon}(u_{\epsilon})=0,\nonumber\\
		&\Rightarrow&\left\langle A(u_{\epsilon}),h\right\rangle+\int_{\Omega}\xi(z)|u_{\epsilon}|^{p-2}u_{\epsilon}hdz=
\int_{\Omega}(u^+)^{p-1}[(u^+)^p+\epsilon]^{\frac{1-(\gamma+p)}{p}}hdz
	\end{eqnarray}
	for all $h\in W^{1,p}(\Omega)$.
	
	In (\ref{eq12}) we choose $h=-u^-_{\epsilon}\in W^{1,p}(\Omega)$. We obtain
	\begin{eqnarray*}
		&&\vartheta(u^-_{\epsilon})=0,\\
		&\Rightarrow&c_1||u^-_{\epsilon}||^p\leq 0\ (\mbox{see Lemma \ref{lem3}}),\\
		&\Rightarrow&u_{\epsilon}\geq 0,\ u_{\epsilon}\neq 0.
	\end{eqnarray*}
	
	From (\ref{eq12}), we have
	\begin{equation}\label{eq13}
		\left\{\begin{array}{l}		-\Delta_pu_{\epsilon}(z)+\xi(z)u_{\epsilon}(z)^{p-1}=u_{\epsilon}(z)^{p-1}[u_{\epsilon}(z)^p+\epsilon]^{\frac{1-(\gamma+p)}{p}}\ \mbox{for almost all}\ z\in\Omega,\\
		\displaystyle		\frac{\partial u_{\epsilon}}{\partial n}=0\ \mbox{on}\ \partial\Omega
		\end{array}\right\}
	\end{equation}
	(see Papageorgiou \& R\u{a}dulescu \cite{18}).
	
	By (\ref{eq13}) and Proposition 7 of Papageorgiou \& R\u{a}dulescu \cite{19}, we have
	$$u_{\epsilon}\in L^{\infty}(\Omega).$$
	
	Then, invoking Theorem 2 of Lieberman \cite{15}, we obtain
	$$u_{\epsilon}\in C_+\backslash\{0\}.$$
	
	From (\ref{eq13}) and hypothesis $H(\xi)$, we have
	\begin{eqnarray*}
		&&\Delta_pu_{\epsilon}(z)\leq||\xi||_{\infty}u_{\epsilon}(z)^{p-1}\ \mbox{for almost all}\ z\in\Omega,\\
		&\Rightarrow&u_{\epsilon}\in D_+\ \mbox{by the nonlinear maximum principle}
	\end{eqnarray*}
	(see Gasinski \& Papageorgiou \cite[p. 738]{3}  and Pucci \& Serrin \cite[p. 120]{25}).
	
	So, for small enough $\epsilon>0$, say $\epsilon\in(0,\epsilon_0)$, we obtain
	 a solution $u_{\epsilon}\in D_+$ for problem (\ref{eq13}).
	\begin{claim}
		$\{u_{\epsilon}\}_{\epsilon\in(0,\epsilon_0)}\subseteq W^{1,p}(\Omega)$ is bounded.
	\end{claim}
	
	We argue by contradiction. So, suppose that the claim is not true. Then we can find\break $\{\epsilon_n\}_{n\geq 1}\subseteq(0,\epsilon_0)$ and corresponding solutions $\{u_n=u_{\epsilon_n}\}_{n\geq 1}\subseteq D_+$ of (\ref{eq13}) such that
	\begin{equation}\label{eq14}
		||u_n||\rightarrow\infty\ \mbox{as}\ n\rightarrow\infty.
	\end{equation}
	
	Let $y_n=\frac{u_n}{||u_n||},\ n\in{\mathbb N}$. Then
	\begin{equation}\label{eq15}
		||y_n||=1\ \mbox{and}\ y_n\geq 0\ \mbox{for all}\ n\in{\mathbb N}.
	\end{equation}
	
	From (\ref{eq12}), we obtain
	\begin{eqnarray}\label{eq16}
		&&\left\langle A(y_n),h\right\rangle+\int_{\Omega}\xi(z)y^{p-1}_nhdz=\int_{\Omega}y^{p-1}_n[u^p_n+\epsilon_n]^{\frac{1-(\gamma+p)}{p}}hdz\\
		&&\mbox{for all}\ h\in W^{1,p}(\Omega),\  n\in{\mathbb N}.\nonumber
	\end{eqnarray}
	
	In (\ref{eq16}) we choose $h=y_n\in W^{1,p}(\Omega)$. Then
	\begin{equation}\label{eq17}
		\vartheta(y_n)=\int_{\Omega}\frac{y^p_n}{[u^p_n+\epsilon_n]^{\frac{p+\gamma-1}{p}}}dz\ \mbox{for all}\ n\in{\mathbb N}.
	\end{equation}
	
	From the first part of the proof, we know that these solutions $u_n$ can be generated by applying the direct method of the calculus of variations to the functionals $\psi_{\epsilon_n}(\cdot)$ and we get
	\begin{eqnarray}\label{eq18}
		&&\psi_{\epsilon_n}(u_n)<0\ \mbox{for all}\ n\in{\mathbb N},\nonumber\\
		&\Rightarrow&\vartheta(u_n)-\frac{p}{1-\gamma}\int_{\Omega}[u^p_n+\epsilon_n]^{\frac{1-\gamma}{p}}dz<0\ \mbox{for all}\ n\in{\mathbb N}.
	\end{eqnarray}
	
	It follows from (\ref{eq17}) and (\ref{eq18}) that
	\begin{eqnarray}\label{eq19}
		\int_{\Omega}\frac{y^p_n}{[u^p_n+\epsilon_n]^{\frac{p+\gamma-1}{p}}}dz&<&\frac{p}{1-\gamma}\int_{\Omega}\frac{[u^p_n+\epsilon_n]^{\frac{1-\gamma}{p}}}{||u_n||^p}dz\nonumber\\
		&\leq&\frac{p}{1-\gamma}\int_{\Omega}\frac{u^{1-\gamma}_n+\epsilon_n^{\frac{1-\gamma}{p}}}{||u_n||^p}dz\rightarrow 0\ \mbox{as}\ n\rightarrow\infty\ (\mbox{see (\ref{eq14})}).
	\end{eqnarray}
	
	Then by
	(\ref{eq17}) and Lemma \ref{lem3}, we have
	\begin{eqnarray*}
		&&c_1||y_n||^p\leq\int_{\Omega}\frac{y^p_n}{[u^p_n+\epsilon_n]^{\frac{p+\gamma-1}{p}}}dz,\\
		&\Rightarrow&y_n\rightarrow 0\ \mbox{in}\ W^{1,p}(\Omega)\ \mbox{as}\ n\rightarrow\infty\ (\mbox{see (\ref{eq19})}),
	\end{eqnarray*}
	which contradicts (\ref{eq15}). This proves the claim.
	
	Consider a sequence $\{\epsilon_n\}_{n\geq 1}\subseteq (0,\epsilon_0)$ such that $\epsilon_n\rightarrow 0^+$. As before, let $\{u_n=u_{\epsilon_n}\}_{n\geq 1}\subseteq D_+$ be the corresponding solutions. On account of the claim, we may assume that
	\begin{equation}\label{eq20}
		u_n\stackrel{w}{\rightarrow}\bar{u}\ \mbox{in}\ W^{1,p}(\Omega)\ \mbox{and}\ u_n\rightarrow \bar{u}\ \mbox{in}\ L^p(\Omega)\ \mbox{as}\ n\rightarrow\infty,\ \bar{u}\geq 0.
	\end{equation}
	
	We know that
	\begin{eqnarray}\label{eq21}
		&&\left\langle A(u_n),h\right\rangle+\int_{\Omega}\xi(z)u^{p-1}_nhdz=\int_{\Omega}
\frac{u^{p-1}_n}{[u^p_n+\epsilon_n]^{\frac{p+\gamma-1}{p}}}hdz\\
		&&\mbox{for all}\ h\in W^{1,p}(\Omega),\  n\in{\mathbb N}\nonumber.
	\end{eqnarray}
	
	Choosing $h=u_n\in W^{1,p}(\Omega)$ in (\ref{eq18}), we obtain
	\begin{equation}\label{eq22}
		-\vartheta(u_n)+\int_{\Omega}\frac{u^p_n}{[u^p_n+\epsilon_n]^{\frac{p+\gamma-1}{p}}}dz=0\ \mbox{for all}\ n\in{\mathbb N}.
	\end{equation}
	
	Moreover, from the first part of the proof (see \eqref{eq11}), we have
	\begin{equation}\label{eq23}
		\vartheta(u_n)-\frac{p}{1-\gamma}\int_{\Omega}[u^p_n+\epsilon_n]^{\frac{1-\gamma}{p}}dz\leq -c_2<0\ \mbox{for all}\ n\in{\mathbb N}.
	\end{equation}
	
	We add (\ref{eq22}) and (\ref{eq23}) and obtain
	\begin{eqnarray}\label{eq24}
		0\leq\int_{\Omega}\frac{u^p_n}{[u^p_n+\epsilon_n]^{\frac{p+\gamma-1}{p}}}dz&\leq&-c_2+\frac{p}{1-\gamma}\int_{\Omega}[u^p_n+\epsilon_n]^{\frac{1-\gamma}{p}}dz\nonumber\\
		&\leq&-c_2+\frac{p}{1-\gamma}\int_{\Omega}[u^{1-\gamma}_n+\epsilon^{\frac{1-\gamma}{n}}_n]dz\ \mbox{for all}\ n\in{\mathbb N}.
	\end{eqnarray}
	
	If $\bar{u}=0$ (see (\ref{eq20})), then
	$$\int_{\Omega}[u^{1-\gamma}_n+\epsilon_n^{\frac{1-\gamma}{n}}]dz\rightarrow 0\ \mbox{as}\ n\rightarrow\infty.$$
	
	This together with (\ref{eq24}) leads to a contradiction. Therefore
	$$\bar{u}\neq 0.$$
	
	On account of (\ref{eq20}) and by passing to a further subsequence if necessary, we may assume that
	\begin{equation}\label{eq25}
		\left.\begin{array}{l}
			u_n(z)\rightarrow\bar{u}(z)\ \mbox{for almost all}\ z\in\Omega\ \mbox{as}\ n\rightarrow\infty,\\
			0\leq u_n(z)\leq k(z)\ \mbox{for almost all}\ z\in\Omega\ \mbox{and all}\ n\in{\mathbb N},\ \mbox{with}\ k\in L^p(\Omega).
		\end{array}\right\}
	\end{equation}
	
	We can always assume that
	\begin{equation}\label{eq26}
		\max\{1,\epsilon_0\}\leq k(z)\ \mbox{for almost all}\ z\in\Omega.
	\end{equation}
	
	For every $n\in{\mathbb N}$, we introduce the following measurable subsets of $\Omega$
	$$\Omega^1_n=\{z\in\Omega:(u_n-\bar{u})(z)>0\}\ \mbox{and}\ \Omega^2_n=\{z\in\Omega:(u_n-\bar{u})(z)<0\},\ n\in{\mathbb N}.$$
	
	Then we have
	\begin{eqnarray}\label{eq27}
		&&\int_{\Omega}\frac{u^{p-1}_n}{[u^p_n+\epsilon_n]^{\frac{p+\gamma-1}{p}}}(u_n-\bar{u})dz\nonumber\\
		&=&\int_{\Omega^1_n}\frac{u^{p-1}_n}{[u^p_n+\epsilon_n]^{\frac{p+\gamma-1}{p}}}(u_n-\bar{u})dz+\int_{\Omega^2_n}\frac{u^{p-1}_n}{[u^p_n+\epsilon_n]^{\frac{p+\gamma-1}{p}}}(u_n-\bar{u})dz\nonumber\\
		&\leq&\int_{\Omega^1_n}\frac{u_n-\bar{u}}{u^{\gamma}_n}dz+\int_{\Omega^2_n}\frac{1}{2k^{\gamma}}\left(\frac{u_n}{k}\right)^{p-1}(u_n-\bar{u})dz\ \mbox{for all}\ n\in{\mathbb N}\ (\mbox{see (\ref{eq25}), (\ref{eq26})}).
	\end{eqnarray}
	
	From (\ref{eq25}) we know that
	\begin{eqnarray}
		&&0\leq\bar{u}(z)\leq k(z)\ \mbox{for almost all}\ z\in\Omega,\label{eq28}\\
		&&-u_n(z)^{-\gamma}\leq-k(z)^{-\gamma}\ \mbox{for almost all}\ z\in\Omega\ \mbox{and all}\ n\in{\mathbb N}.\label{eq29}
	\end{eqnarray}
	
	It follows from (\ref{eq28}), (\ref{eq29}) 
	 that
	\begin{equation}\label{eq30}
		-\bar{u}(z)u_n(z)^{-\gamma}\leq-k(z)^{1-\gamma}\ \mbox{for almost all}\ z\in\Omega\ \mbox{and all}\ n\in{\mathbb N}.
	\end{equation}
	
	Then for all $n\in{\mathbb N}$ we have
	\begin{eqnarray}\label{eq31}
		&&\int_{\Omega^1_n}\frac{u_n-\bar{u}}{u^{\gamma}_n}dz=
\int_{\Omega^1_n}[u^{1-\gamma}_n-\bar{u}u^{-\gamma}_n]dz\nonumber\\
		&&\mbox{for all }n\in{\mathbb N}\ (\mbox{see (\ref{eq25}), (\ref{eq30})}),\nonumber\\
		&\Rightarrow&\limsup\limits_{n\rightarrow\infty}\int_{\Omega^1_n}\frac{u_n-\bar{u}}{u^{\gamma}_n}dz\leq 0.
	\end{eqnarray}
	
	Also, from (\ref{eq25}) and (\ref{eq20}), we can see that
	\begin{equation}\label{eq32}
		\int_{\Omega^2_n}\frac{1}{2k^{\gamma}}\left(\frac{u_n}{k}\right)^{p-1}(u_n-\bar{u})dz\rightarrow 0\ \mbox{as}\ n\rightarrow\infty.
	\end{equation}
	
	We return to (\ref{eq27}), pass to the limit as $n\rightarrow\infty$, and use (\ref{eq31}) and (\ref{eq32}). We obtain
	\begin{equation}\label{eq33}
		\limsup\limits_{n\rightarrow\infty}\int_{\Omega}\frac{u_n^{p-1}}{[u^p_n+\epsilon_n]^{\frac{p+\gamma-1}{p}}}(u_n-\bar{u})dz\leq 0.
	\end{equation}
	
	In (\ref{eq21}) we choose $h=u_n-\bar{u}\in W^{1,p}(\Omega)$. Then
	\begin{eqnarray}\label{eq34}
		&&\left\langle A(u_n),u_n-\bar{u}\right\rangle+\int_{\Omega}\xi(z)u_n^{p-1}(u_n-\bar{u})dz=\int_{\Omega}\frac{u_n^{p-1}}{[u^p_n+\epsilon_n]^{\frac{p+\gamma-1}{p}}}(u_n-\bar{u})dz\nonumber\\
		&&\mbox{for all}\ n\in{\mathbb N},\nonumber\\
		&\Rightarrow&\limsup\limits_{n\rightarrow\infty}\left\langle A(u_n),u_n-\bar{u}\right\rangle\leq 0\ \mbox{(see (\ref{eq20}), (\ref{eq33}))},\nonumber\\
		&\Rightarrow&u_n\rightarrow\bar{u}\ \mbox{in}\ W^{1,p}(\Omega)\ (\mbox{see Proposition \ref{prop2}}),\ \bar{u}\geq 0,\ \bar{u}\neq 0.
	\end{eqnarray}
	
Using in \eqref{eq12} as a test function
$$h=\frac{u_n^{p-1}}{(u_n^p+\epsilon_n)^{\frac{p+\gamma-1}{p}\frac{p'}{p}}}\in W^{1,p}(\Omega)$$
(recall that $u_n\in D_+$) and our hypothesis on $\gamma$, we
can
 infer that
$$\left\{\frac{u^{p-1}_n}{(u^p_n+\epsilon_n)^{\frac{p+\gamma-1}{p}}}\right\}_{n\geq 1}\subseteq L^{p'}(\Omega)\ \mbox{is bounded}.$$
	
	Also, we have
	$$\frac{u^p_n}{(u^{p-1}_n+\epsilon_n)^{\frac{p+\gamma-1}{p}}}\rightarrow\bar{u}^{-\gamma}\ \mbox{for almost all}\ z\in\Omega\ (\mbox{see (\ref{eq25})}).$$
	
	Then Problem 1.19 in Gasinski \& Papageorgiou \cite[p. 46]{5} implies that
	\begin{eqnarray}\label{eq35}
		&&\frac{u^{p-1}_n}{(u^p_n+\epsilon_n)^{\frac{p+\gamma-1}{p}}}\stackrel{w}{\rightarrow}\bar{u}^{-\gamma}\ \mbox{in}\ L^{p'}(\Omega),\nonumber\\
		&\Rightarrow&\int_{\Omega}\frac{u^{p-1}_n}{[u^p_n+\epsilon_n]^{\frac{p+\gamma-1}{p}}}hdz\rightarrow\int_{\Omega}\bar{u}^{-\gamma}hdz\ \mbox{for all}\ h\in W^{1,p}(\Omega).
	\end{eqnarray}
	
	Passing to the limit as $n\rightarrow\infty$ in (\ref{eq21}) and using (\ref{eq34}) and (\ref{eq35}), we obtain
	\begin{equation}\label{eq36}
		\left\langle A(\bar{u}),h\right\rangle+\int_{\Omega}\xi(z)\bar{u}^{p-1}hdz=\int_{\Omega}\bar{u}^{-\gamma}hdz\ \mbox{for all}\ h\in W^{1,p}(\Omega)
	\end{equation}
	
	In (\ref{eq36}) we first choose $h=\frac{1}{[\bar{u}^p+\delta]^{\frac{p-1}{p}}}\in W^{1,p}(\Omega),\delta>0$. Then
	\begin{eqnarray*}
		&&\int_{\Omega}\xi(z)\frac{\bar{u}^{p-1}}{[\bar{u}^p+\delta]^{\frac{p-1}{p}}}dz\geq\int_{\Omega}\frac{\bar{u}^{-\gamma}}{[\bar{u}^p+\delta]^{\frac{p-1}{p}}}dz,\\
		&&\int_{\Omega}\frac{\bar{u}^{-\gamma}}{[\bar{u}^p+\delta]^{\frac{p-1}{p}}}dz\leq||\xi||_{\infty}|\Omega|_N\ (\mbox{see hypothesis}\ H(\xi)).
	\end{eqnarray*}
	
	We let $\delta\rightarrow 0^+$ and use Fatou's lemma. Then
	\begin{equation}\label{eq37}
		\int_{\Omega}\frac{1}{\bar{u}^{p+\gamma-1}}dz\leq||\xi||_{\infty}|\Omega|_N.
	\end{equation}
	
	Next, we choose in (\ref{eq36}) $h=\frac{1}{[\bar{u}^p+\gamma]^{\frac{2(p-1)+\gamma}{p}}}\in W^{1,p}(\Omega)$. Reasoning as above, we obtain 
	 via Fatou's lemma as $\delta\rightarrow 0^+$
	\begin{eqnarray*}
		\int_{\Omega}\frac{\bar{u}^{-\gamma}}{\bar{u}^{2(p-1)+\gamma}}dz=\int_{\Omega}\frac{1}{\bar{u}^{2(p+\gamma-1)}}dz&\leq&\int_{\Omega}\xi(z)\frac{\bar{u}^{p-1}}{\bar{u}^{2(p-1)+\gamma}}dz\\
		&=&\int_{\Omega}\xi(z)\frac{1}{\bar{u}^{p+\gamma-1}}dz\\
		&\leq&||\xi||^2_{\infty}|\Omega|_N\ (\mbox{see (\ref{eq37})}).
	\end{eqnarray*}
	
	Continuing in this way, we obtain
	\begin{equation}\label{eq38}
		\int_{\Omega}\frac{1}{\bar{u}^{k(p+\gamma-1)}}dz\leq||\xi||^k_{\infty}|\Omega|_N\ \mbox{for all}\ k\in{\mathbb N}.
	\end{equation}
	
	Therefore we can infer that
	\begin{eqnarray*}
		&&\bar{u}^{-(p+\gamma-1)}\in L^{\tau}(\Omega)\ \mbox{for all}\ \tau\geq 1,\\
		&&\limsup\limits_{\tau\rightarrow+\infty}||\bar{u}^{-(p+\gamma-1)}||_{\tau}<+\infty.
	\end{eqnarray*}
	
	Then Problem 3.104 in Gasinski \& Papageorgiou \cite[p. 477]{4} implies that
	$$\bar{u}^{-(p+\gamma-1)}\in L^{\infty}(\Omega).$$
	
	Note that
	$$\bar{u}^{-\gamma}=\bar{u}^{-(p+\gamma-1)}\bar{u}^{p-1}.$$
	
	Therefore from (\ref{eq36}) and Proposition 7 of Papageorgiou \& R\u{a}dulescu \cite{19}, we have
	$$\bar{u}\in L^{\infty}(\Omega).$$
	
	Invoking Theorem 2 of Lieberman \cite{15}, we have
	$$\bar{u}\in C_+\backslash\{0\}.$$
	
	It follows by (\ref{eq36})  that
	\begin{eqnarray}\label{eq39}
		&&-\Delta_p\bar{u}(z)+\xi(z)\bar{u}(z)^{p-1}=\bar{u}(z)^{-\gamma}\ \mbox{for almost all}\ z\in\Omega,\ \frac{\partial\bar{u}}{\partial n}=0\ \mbox{on}\ \partial\Omega\\
		&&(\mbox{see Papageorgiou \& R\u{a}dulescu \cite{18}}),\nonumber\\
		&\Rightarrow&\Delta_p\bar{u}(z)\leq||\xi||_{\infty}\bar{u}(z)^{p-1}\ \mbox{for almost all}\ z\in\Omega,\nonumber\\
		&\Rightarrow&\bar{u}\in D_+\ (\mbox{by the nonlinear maximum principle (see (\cite[p. 738]{3} and \cite[p. 120]{25}))}).\nonumber
	\end{eqnarray}
	
	Finally, we
	can  show that the positive solution is unique.	
	Suppose that $\bar{u}_0\in W^{1,p}(\Omega)$ is another positive solution of (\ref{eq7}). Again we have $\bar{u}_0\in D_+$. Also
	\begin{eqnarray*}
		&&0\leq\left\langle A(\bar{u})-A(\bar{u}_0),\bar{u}-u_0\right\rangle+\int_{\Omega}\xi(z)(\bar{u}^{p-1}-\bar{u}_0^{p-1})(\bar{u}-\bar{u}_0)dz\\
		&&=\int_{\Omega}(\bar{u}^{-\gamma}-\bar{u}_0^{-\gamma})(\bar{u}-\bar{u}_0)dz\leq 0,\\
		&\Rightarrow&\bar{u}=\bar{u}_0\ (\mbox{the function}\ x\mapsto\frac{1}{x^{\gamma}}\ \mbox{is strictly decreasing on}\ (0,+\infty)).
	\end{eqnarray*}
	
	This proves the uniqueness of the positive solution $\bar{u}\in D_+$ of (\ref{eq7})
	and thus completes the proof of  Proposition~\ref{prop4}.
\end{proof}

\section{Existence of positive solutions}

Let $\bar{u}\in D_+$ be the unique positive solution of (\ref{eq7}) produced by Proposition \ref{prop4}. We choose $t\in(0,1)$ small enough such that
\begin{equation}\label{eq40}
	\tilde{u}=t\bar{u}\leq \min\{\hat{c},\delta_0\}\ \mbox{on}\ \overline{\Omega}\ (\mbox{see hypotheses $H(f)(i),(ii)$}).
\end{equation}

Then given $\tilde{v}\in W^{1,p}(\Omega)$, we have
\begin{eqnarray}\label{eq41}
	-\Delta_p\tilde{u}(z)+\xi(z)\tilde{u}(z)^{p-1}&=&t^{p-1}[-\Delta_p\bar{u}(z)+\xi(z)\bar{u}(z)^{p-1}]\nonumber\\
	&=&t^{p-1}\bar{u}(z)^{-\gamma}\ (\mbox{see (\ref{eq39})})\nonumber\\
	&\leq&\tilde{u}(z)^{-\gamma}+f(z,\tilde{u}(z),Dv(z))\ \mbox{for almost all}\ z\in\Omega
\end{eqnarray}
(see (\ref{eq40}) and hypothesis $H(f)(ii)$).

Given $v\in C^1(\overline{\Omega})$, we consider the following nonlinear auxiliary Neumann problem:
\begin{equation}\label{eq42}
	\left\{\begin{array}{l}
		-\Delta_pu(z)+\xi(z)u(z)^{p-1}=u(z)^{-\gamma}+f(z,u(z),Dv(z))\ \mbox{in}\ \Omega,\\
		\frac{\partial u}{\partial n}=0\ \mbox{on}\ \partial\Omega,\ u>0.
	\end{array}\right\}
\end{equation}
\begin{proposition}\label{prop5}
	If hypotheses $H(\xi),H(f)$ hold, then for every $v\in C^1(\overline{\Omega})$ problem (\ref{eq42}) has a solution $u_v\in[\tilde{u},w]\cap C^1(\overline{\Omega})$, with $w(\,\cdot\,)$ being the function from hypothesis $H(f)(i)$.
\end{proposition}
\begin{proof}
	We introduce the following truncation of the reaction
	term  in problem (\ref{eq1}):
	\begin{eqnarray}\label{eq43}
		\hat{f}_v(z,x)=\left\{\begin{array}{l}
			\tilde{u}(z)^{-\gamma}+f(z,\tilde{u}(z),Dv(z))\ \mbox{if}\ x<\tilde{u}(z)\\
			x^{-\gamma}+f(z,x,Dv(z))\ \mbox{if}\ \tilde{u}(z)\leq x\leq w(z)\\
			w(z)^{-\gamma}+f(z,w(z),Dv(z))\ \mbox{if}\ w(z)<x.
		\end{array}\right.
	\end{eqnarray}
	
	Evidently, $\hat{f}_v(\cdot,\cdot)$ is a Carath\'eodory function. We set $\hat{F}_v(z,x)=\int^x_0\hat{f}_v(z,s)ds$ and consider the $C^1$-functional $\hat{\varphi}_v:W^{1,p}(\Omega)\rightarrow\RR$ defined by
	$$\hat{\varphi}_v(u)=\frac{1}{p}\vartheta(u)-\int_{\Omega}\hat{F}_v(z,u(z))dz\ \mbox{for all}\ u\in W^{1,p}(\Omega).$$
	
	It is clear from (\ref{eq43}) that $\hat{\varphi}_v(\cdot)$ is coercive. Also, it is sequentially weakly lower semicontinuous. So, by the Weierstrass-Tonelli theorem, we can find $u_v\in W^{1,p}(\Omega)$ such that
	\begin{eqnarray}\label{eq44}
		&&\hat{\varphi}_v(u_v)=\inf\{\hat{\varphi}_v(u):u\in W^{1,p}(\Omega)\},\nonumber\\
		&\Rightarrow&\hat{\varphi}'_v(u_v)=0,\nonumber\\
		&\Rightarrow&\left\langle A(u_v),h\right\rangle+\int_{\Omega}\xi(z)|u_v|^{p-2}u_vhdz=\int_{\Omega}\hat{f}_v(z,u_v)hdz\ \mbox{for all}\ h\in W^{1,p}(\Omega).
	\end{eqnarray}
	
	In (\ref{eq44})  we first choose $h=(\tilde{u}-u_v)^+\in W^{1,p}(\Omega)$. We have
	\begin{eqnarray}\label{eq45}
		&&\left\langle A(u_v),(\tilde{u}-u_v)^+\right\rangle+\int_{\Omega}\xi(z)|u_v|^{p-2}u_v(\tilde{u}-u_v)^+dz\nonumber\\
		&&=\int_{\Omega}[\tilde{u}^{-\gamma}+f(z,\tilde{u},Dv)](\tilde{u}-u_v)^+dz\ (\mbox{see (\ref{eq43})})\nonumber\\
		&&\geq\left\langle A(\tilde{u}),(\tilde{u}-u_v)^+\right\rangle+\int_{\Omega}\xi(z)\tilde{u}^{p-1}(\tilde{u}-u_v)^+dz\ (\mbox{see (\ref{eq41})}),\nonumber\\
		&\Rightarrow&0\geq\left\langle A(\tilde{u})-A(u_v),(\tilde{u}-u_v)^+\right\rangle+\int_{\Omega}\xi(z)(\tilde{u}^{p-1}-|u_v|^{p-2}u_v)(\tilde{u}-u_v)^+dz,\nonumber\\
		&\Rightarrow&\tilde{u}\leq u_v.
	\end{eqnarray}
	
	Next, we choose in (\ref{eq44}) $h=(u_v-w)^+\in W^{1,p}(\Omega)$. Then
	\begin{eqnarray}\label{eq46}
		&&\left\langle A(u_v),(u_v-w)^+\right\rangle+\int_{\Omega}\xi(z)u_v^{p-1}(u_v-w)^+dz\ (\mbox{see (\ref{eq45})})\nonumber\\
		&&=\int_{\Omega}[w^{-\gamma}+f(z,w,Dv)](u_v-w)^+dz\ (\mbox{see (\ref{eq43})})\nonumber\\
		&&\leq\left\langle A(w),(u_v-w)^+\right\rangle+\int_{\Omega}\xi(z)w^{p-1}(u_v-w)^+dz\ (\mbox{see hypothesis}\ H(f)(i)),\nonumber\\
		&\Rightarrow&\left\langle A(u_v)-A(w),(u_v-w)^+\right\rangle+\int_{\Omega}\xi(z)(u_v^{p-1}-w^{p-1})(u_v-w)^+dz\leq 0,\nonumber\\
		&\Rightarrow&u_v\leq w.
	\end{eqnarray}
	
	It follows from (\ref{eq45}) and (\ref{eq46}) that
	\begin{equation}\label{eq47}
		u_v\in[\tilde{u},w].
	\end{equation}
	
	On account of (\ref{eq47}), (\ref{eq43}) and (\ref{eq44}), we have
	\begin{eqnarray}\label{eq48}
		&&-\Delta_pu_v(z)+\xi(z)u_v(z)^{p-1}=u_v(z)^{-\gamma}+f(z,u_v(z),Dv(z))\ \mbox{for almost all}\ z\in\Omega,\nonumber\\
		&&\frac{\partial u_v}{\partial n}=0\ \mbox{on}\ \partial\Omega\\
		&&(\mbox{see Papageorgiou \& R\u{a}dulescu \cite{18}}).\nonumber
	\end{eqnarray}
	
	From (\ref{eq48}) and  Papageorgiou \& R\u{a}dulescu \cite[Proposition 7]{19}, we have
	$$u_v\in L^{\infty}(\Omega).$$
	
	Then Theorem 2 of Lieberman \cite{15} implies that $u_v\in D_+$. Therefore
	$$u_v\in[\tilde{u},w]\cap C^1(\overline{\Omega}).$$
The proof 
of Proposition~\ref{prop5}
is now complete.
\end{proof}

We introduce the solution set
$$S_v=\{u\in W^{1,p}(\Omega):u\ \mbox{is a solution of (\ref{eq42})},\ u\in[\tilde{u},w]\}.$$

By Proposition \ref{prop5}, we have
$$\emptyset\neq S_v\subseteq[\tilde{u},v]\cap C^1(\overline{\Omega}).$$

In fact, we have the following stronger result for the elements of $S_v$.
\begin{proposition}\label{prop6}
	If hypotheses $H(\xi),\, H(f)$ hold and $u\in S_v$, then $u\in {\rm int}_{C^1(\overline{\Omega})}[\tilde{u},w]$.
\end{proposition}
\begin{proof}
	Let $\tilde{\rho}=\min\limits_{\overline{\Omega}}\tilde{u}>0$ (recall that $\tilde{u}\in D_+$). So, we can increase $\hat{\xi}_{\rho}>0$ postulated by hypothesis $H(f)(iii)$ in order to guarantee that for almost all $z\in\Omega$, the function
	$$x\mapsto x^{-\gamma}+f(z,x,Dv(z))+\hat{\xi}_px^{p-1}$$
	is nondecreasing on $[\tilde{\rho},\rho]\subseteq\RR_+$.
	
	Let $\delta>0$ and set $\tilde{u}^{\delta}=\tilde{u}+\delta\in D_+$. Then
	\begin{eqnarray*}
		&&-\Delta_p\tilde{u}^{\delta}+(\xi(z)+\hat{\xi}_{\rho})(\tilde{u}^{\delta})^{p-1}\\
		&&\leq-\Delta_p\tilde{u}+(\xi(z)+\hat{\xi}_{\rho})\tilde{u}^{p-1}+\lambda(\delta)\ \mbox{with}\ \lambda(\delta)\rightarrow 0^+\ \mbox{as}\ \delta\rightarrow 0^+\\
		&&\leq\tilde{u}^{-\gamma}+f(z,\tilde{u},Dv)+\hat{\xi}_{\rho}\tilde{u}^{p-1}\ \mbox{for}\ \delta>0\ \mbox{small enough}\\
		&&(\mbox{since}\ f(z,\tilde{u},Dv)\geq\tilde{c}_{\tilde{\rho}}>0\ \mbox{for almost all}\ z\in\Omega,\ \mbox{see}\ H(f)(i))\\
		&&\leq u^{-\gamma}+f(z,u,Dv)+\hat{\xi}_{\rho}u^{p-1}\ (\mbox{since}\ \tilde{u}\leq u)\\
		&&=-\Delta_pu+(\xi(z)+\hat{\xi}_{\rho})u^{p-1}\ \mbox{for almost all}\ z\in\Omega\ (\mbox{since}\ u\in S_v),\\
		&\Rightarrow&\tilde{u}^{\delta}\leq u\ \mbox{for small enough}\ \delta>0,\\
		&\Rightarrow&u-\tilde{u}\in D_+.
	\end{eqnarray*}
	
	Similarly, for $\delta>0$ let $u^{\delta}=u+\delta\in D_+$. Then
	\begin{eqnarray*}
		&&-\Delta_pu^{\delta}+(\xi(z)+\hat{\xi}_{\rho})(u^{\delta})^{p-1}\\
		&&\leq-\Delta_pu+(\xi(z)+\hat{\xi}_{\rho})u^{p-1}+\tilde{\lambda}(\lambda)\ \mbox{with}\ \tilde{\lambda}(\delta)\rightarrow 0^+\ \mbox{as}\ \delta\rightarrow 0^+\\
		&&=u^{-\gamma}+f(z,u,Dv)+\hat{\xi}_{\rho}u^{p-1}+\tilde{\lambda}(\delta)\ (\mbox{since}\ u\in S_v)\\
		&&\leq w^{-\gamma}+f(z,w,Dv)+\hat{\xi}_{\rho}u^{p-1}+\tilde{\lambda}(\delta)\ (\mbox{since}\ u\leq w)\\
		&&\leq-c^*+\tilde{\lambda}(\delta)+\hat{\xi}_{\rho}u^{p-1}\ (\mbox{see hypothesis}\ H(f)(i))\\
		&&\leq-\Delta_pw+(\xi(z)+\hat{\xi}_p)w^{p-1}\ \mbox{for almost all}\ z\in\Omega\ \mbox{and for small enough}\ \delta>0\\
		&&(\mbox{since}\ \tilde{\lambda}(\delta)\rightarrow 0^+\ \mbox{as}\ \delta\rightarrow 0^+\ \mbox{and due to hypothesis}\ H(f)(i)),\\
		&\Rightarrow&u^{\delta}\leq w\ \mbox{for small enough}\ \delta>0,\\
		&\Rightarrow&(w-u)(z)>0\ \mbox{for all}\ z\in\overline{\Omega}.
	\end{eqnarray*}
	
	Therefore we conclude that
	$$u\in {\rm int}_{C^1(\overline{\Omega})}[\tilde{u},w].$$
The proof 
of Proposition~\ref{prop6}
is now complete.
\end{proof}

We can show that $S_v$ admits a smallest element, that is, there exists $\hat{u}_v\in S_v$ such that $\hat{u}_v\leq u$ for all $u\in S_v$.
\begin{proposition}\label{prop7}
	If hypotheses $H(\xi),\,H(f)$ hold, then for every $v\in C^1(\overline{\Omega})$, the solution set $S_v$ admits a smallest element
	$$\hat{u}_v\in S_v.$$
\end{proposition}
\begin{proof}
	Invoking Lemma 3.10 in Hu \& Papageorgiou \cite[p. 178]{13}, we can find a sequence \\$\{u_n\}_{n\geq 1}\subseteq S_v$ such that
	$${\rm essinf}\, S_v=\inf\limits_{n\geq 1}u_n.$$
	
	For every $n\in{\mathbb N}$, we have
	\begin{eqnarray}
		&&\left\langle A(u_n),h\right\rangle+\int_{\Omega}\xi(z)u_n^{p-1}hdz=\int_{\Omega}[u_n^{-\gamma}+f(z,u_n,Dv)]hdz\label{eq49}\\
		&&\mbox{for all}\ h\in W^{1,p}(\Omega),\ n\in{\mathbb N},\nonumber\\
		&&\tilde{u}\leq u_n\leq w\ \mbox{for all}\ n\in{\mathbb N}.\label{eq50}
	\end{eqnarray}
	
	It follows from (\ref{eq49}) and (\ref{eq50}) that
	$$\{u_n\}_{n\geq 1}\subseteq W^{1,p}(\Omega)\ \mbox{is bounded}.$$
	
	So, we may assume that
	\begin{eqnarray}\label{eq51}
		u_n\stackrel{w}{\rightarrow}\hat{u}_v\ \mbox{in}\ W^{1,p}(\Omega)\ \mbox{and}\ u_n\rightarrow\hat{u}_v\ \mbox{in}\ L^p(\Omega)\ \mbox{as}\ n\rightarrow\infty,\ \hat{u}_v\in[\tilde{u},w].
	\end{eqnarray}
	
	In (\ref{eq49}) we choose $h=u_n-\hat{u}_v\in W^{1,p}(\Omega)$, pass to the limit as $n\rightarrow\infty$, and use (\ref{eq51}). Then
	\begin{eqnarray}\label{eq52}
		&&\lim\limits_{n\rightarrow\infty}\left\langle A(u_n),u_n-\hat{u}_v\right\rangle=0\ \mbox{see (\ref{eq50})},\nonumber\\
		&\Rightarrow&u_n\rightarrow\hat{u}_v\ \mbox{in}\ W^{1,p}(\Omega)\ (\mbox{see Proposition \ref{prop2}}).
	\end{eqnarray}
	
	Therefore, if in (\ref{eq49}) we pass to the limit as $n\rightarrow\infty$ and use (\ref{eq52}), then
	\begin{eqnarray*}
		&&\left\langle A(\hat{u}_v),h\right\rangle+\int_{\Omega}\xi(z)\hat{u}_v^{p-1}hdz=\int_{\Omega}[\hat{u}_v^{-\gamma}+f(z,\hat{u}_v,Dv)]hdz\\
		&&\mbox{for all}\ h\in W^{1,p}(\Omega),\\
		&\Rightarrow&\hat{u}_v\in S_v\subseteq D_+\ \mbox{and }{\rm essinf}\, S_v=\hat{u}_v.
	\end{eqnarray*}
The proof of Proposition \ref{prop7} is now complete.
\end{proof}

We can define a map $\sigma:C^1(\overline{\Omega})\rightarrow C^1(\overline{\Omega})$ by
$$\sigma(v)=\hat{u}_v.$$

This map is well-defined by Proposition \ref{prop7} and any fixed point of $\sigma(\cdot)$ is a solution of problem (\ref{eq1}). To generate a fixed point for $\sigma(\cdot)$, we will use Theorem \ref{th1} (the Leray-Schauder alternative principle). For this purpose, the next lemma will be useful.

\begin{lemma}\label{lem8}
	If hypotheses $H(\xi),\,H(f)$ hold, $\{v_n\}_{n\geq 1}\subseteq C^1(\overline{\Omega})$, $v_n\rightarrow v$ in $C^1(\overline{\Omega})$, and $u\in S_v$, then for every
	$n\in{\mathbb N}$ there exists $u_n\in S_{v_n}$ such that $u_n\rightarrow u$ in $C^1(\overline{\Omega})$.
\end{lemma}
\begin{proof}
	We consider the following nonlinear Neumann problem
	\begin{equation}\label{eq53}
		\left\{\begin{array}{l}
			-\Delta_py(z)+\xi(z)|y(z)|^{p-2}y(z)=u(z)^{-\gamma}+f(z,u(z),Dv_n(z))\ \mbox{in}\ \Omega,\\
	\displaystyle			\frac{\partial y}{\partial n}=0\ \mbox{on}\ \partial\Omega.
		\end{array}\right\}
	\end{equation}
	
	Since $u\in S_v\subseteq D_+$, we have
	\begin{equation}\label{eq54}
		\left\{\begin{array}{l}
			k_n(z)=u(z)^{-\gamma}+f(z,u(z),Dv_n(z))\geq 0\ \mbox{for almost all}\ z\in\Omega\ \mbox{and all}\ n\in{\mathbb N},\\
			\{k_n\}_{n\geq 1}\subseteq L^{\infty}(\Omega)\ \mbox{is bounded},k_n\neq 0\ \mbox{for all}\ n\in{\mathbb N}\\
			(\mbox{see hypotheses}\ H(f)(i),(ii)).
		\end{array}\right\}
	\end{equation}
	
	In problem (\ref{eq53}), the left-hand side determines a maximal monotone coercive operator (see Lemma \ref{lem3}), which is strictly monotone. Therefore, on account of (\ref{eq54}), problem (\ref{eq53}) admits a unique solution $y_n^0\in W^{1,p}(\Omega)$, $y^0_n\neq 0$. We have for all $n\in{\mathbb N}$
	\begin{eqnarray}\label{eq55}
		&&\left\langle A(y^0_n),h\right\rangle+\int_{\Omega}\xi(z)|y^0_n|^{p-2}y^0_nhdz=\int_{\Omega}k_n(z)hdz\ \mbox{for all}\ h\in W^{1,p}(\Omega).
	\end{eqnarray}
	
	In (\ref{eq55}) we choose $h=-(y^0_n)^-\in W^{1,p}(\Omega)$. Then
	\begin{eqnarray*}
		&&\vartheta((y^0_n)^-)\leq 0\ (\mbox{see (\ref{eq54})}),\\
		&\Rightarrow&c_1||(y^0_n)^-||^p\leq 0\ (\mbox{see Lemma \ref{lem3}}),\\
		&\Rightarrow&y^0_n\geq 0,\ y^0_n\neq 0\ \mbox{for all}\ n\in{\mathbb N}.
	\end{eqnarray*}
	
	Also, it is clear from (\ref{eq54}) and (\ref{eq55}) that
	$$\{y^0_n\}_{n\geq 1}\subseteq W^{1,p}(\Omega)\ \mbox{is bounded.}$$
	
	Invoking Proposition 7 of Papageorgiou \& R\u{a}dulescu \cite{19}, we have
	\begin{equation}\label{eq56}
		y^0_n\in L^{\infty}(\Omega)\ \mbox{and}\ ||y^0_n||_{\infty}\leq c_5\ \mbox{for some}\ c_5>0\ \mbox{and all}\ n\in{\mathbb N}.
	\end{equation}
	
	Then (\ref{eq53}) and Theorem 2 of Lieberman \cite{15} imply that there exist $\alpha\in(0,1)$ and $c_6>0$ such that
	\begin{equation}\label{eq57}
		y^0_n\in C^{1,\alpha}(\overline{\Omega})\ \mbox{and}\ ||y^0_n||_{C^{1,\alpha}(\overline{\Omega})}\leq c_6\ \mbox{for all}\ n\in{\mathbb N}.
	\end{equation}
	
	Recall that $C^{1,\alpha}(\overline{\Omega})$ is  compactly
	embedded in $C^1(\overline{\Omega})$. So, from (\ref{eq57}) we see that we can find a subsequence $\{y^0_{n_k}\}_{k\geq 1}$ of $\{y^0_n\}_{n\geq 1}$ such that
	\begin{equation}\label{eq58}
		y^0_{n_k}\rightarrow y^0\ \mbox{in}\ C^1(\overline{\Omega})\ \mbox{as}\ k\rightarrow\infty,\ y^0\geq 0.
	\end{equation}
	
	Note that
	\begin{equation}\label{eq59}
		k_n\rightarrow k\ \mbox{in}\ L^{p'}(\Omega)\ \mbox{with}\ k(z)=u(z)^{-\gamma}+f(z,u(z),Dv(z)).
	\end{equation}
	
	Using (\ref{eq55}) (for the $y^0_{n_k}$'s) and (\ref{eq58}), (\ref{eq59}), we obtain
	\begin{eqnarray}\label{eq60}
		&&\left\langle A(y^0),h\right\rangle+\int_{\Omega}\xi(z)(y^0)^{p-1}hdz=\int_{\Omega}k(z)hdz\ \mbox{for all}\ h\in W^{1,p}(\Omega),\nonumber\\
		&\Rightarrow&-\Delta_py^0(z)+\xi(z)y^0(z)^{p-1}=u(z)^{-\gamma}+f(z,u(z),Dv(z))\ \mbox{for almost all}\ z\in\Omega,\\
		&&\frac{\partial y^0}{\partial n}=0\ \mbox{on}\ \partial\Omega.\nonumber
	\end{eqnarray}
	
	Problem (\ref{eq60}) admits a unique solution. Since $u\in S_v$, $u$ solves (\ref{eq60}) and so $y^0=u$. Therefore for the initial sequence we have
	\begin{equation}\label{eq61}
		y^0_n\rightarrow u\ \mbox{in}\ C^1(\overline{\Omega})\ \mbox{as}\ n\rightarrow\infty.
	\end{equation}
	
	Next, we consider the following nonlinear Neumann problem
	\begin{eqnarray*}
		\left\{\begin{array}{l}
			-\Delta_py(z)+\xi(z)|y(z)|^{p-2}y(z)=y^0_n(z)^{-\gamma}+f(z,y^0_n(z),Dv_n(z))\ \mbox{in}\ \Omega,\\
	\displaystyle			\frac{\partial y}{\partial n}=0\ \mbox{on}\ \partial\Omega.
		\end{array}\right\}
	\end{eqnarray*}
	
	Evidently, this problem has a unique solution $y^1_n\in D_+$ and
	$$y^1_n\rightarrow u\ \mbox{in}\ C^1(\overline{\Omega})\ \mbox{as}\ n\rightarrow\infty\ (\mbox{see (\ref{eq61})}).$$
	
	Continuing in this way, we produce a sequence $\{y^k_n\}_{k,n\in{\mathbb N}}$ such that
	\begin{equation}\label{eq62}
		\left\{\begin{array}{l}
			-\Delta_py^k_n(z)+\xi(z)y^k_n(z)^{p-1}=y^{k-1}_n(z)^{-\gamma}+f(z,y^{k-1}_n(z),Dv_n(z))\\
			 \mbox{for almost all}\ z\in\Omega,\\
		\displaystyle		\frac{\partial u^k_n}{\partial n}=0\ \mbox{on}\ \partial\Omega,\ k,n\in{\mathbb N}
		\end{array}\right\}
	\end{equation}
	\begin{equation}\label{eq63}
		\mbox{and}\ y^k_n\rightarrow u\ \mbox{in}\ C^1(\overline{\Omega})\ \mbox{as}\ n\rightarrow\infty\ \mbox{for all}\ k\in{\mathbb N}.
	\end{equation}
	
	From (\ref{eq59}), (\ref{eq60}) and Theorem 2 of Lieberman \cite{15}, we can deduce as before that
	$$\{y^k_n\}_{k\in {\mathbb N}}\subseteq C^1(\overline{\Omega})\ \mbox{is relatively compact.}$$
	
	So, we can find a subsequence $\{y^{k_m}_n\}_{m\in{\mathbb N}}$ of $\{y^k_n\}_{k\in{\mathbb N}}$ ($n\in{\mathbb N}$ is fixed) such that
	$$y^{k_m}_n\rightarrow\hat{y}_n\ \mbox{in}\ C^1(\overline{\Omega}),\ n\in{\mathbb N}.$$
	
	From (\ref{eq62}) in the limit we obtain
	\begin{equation}\label{eq64}
		\left\{\begin{array}{l}
			-\Delta_p\hat{y}_n(z)+\xi(z)\hat{y}_n(z)^{p-1}=\hat{y}_n(z)^{-\gamma}+f(z,\hat{y}_n(z),Dv_n(z))\ \mbox{for almost all}\ z\in\Omega,\\
	\displaystyle			\frac{\partial\hat{y}_n}{\partial n}=0\ \mbox{on}\ \partial\Omega.
		\end{array}\right\}
	\end{equation}
	
	Then, using Theorem 2 of Lieberman \cite{15} as before, \eqref{eq63} and the double limit lemma (see Gasinski \& Papageorgiou \cite[Problem 1.175, p. 61]{4}) we obtain
	\begin{eqnarray*}
		&&\hat{y}_n\rightarrow u\ \mbox{in}\ C^1(\overline{\Omega})\ \mbox{as}\ n\rightarrow\infty,\\
		&\mbox{and}&\hat{y}_n\in S_{v_n}\ \mbox{for}\ n\geq n_0\ (\mbox{see Proposition \ref{prop6}}).
	\end{eqnarray*}
The proof 
of Lemma~\ref{lem8}
is now complete.
\end{proof}

Using this lemma we can show that the minimal solution map $\sigma(\cdot)$ is compact.
\begin{proposition}\label{prop9}
	If hypotheses $H(\xi),\,H(f)$ hold, then the minimal solution map $\sigma:C^1(\overline{\Omega})\rightarrow C^1(\overline{\Omega})$ defined by $\sigma(v)=\hat{u}_v$ is compact.
\end{proposition}
\begin{proof}
	We first show that $\sigma(\cdot)$ is continuous. To this end, let $v_n\rightarrow v$ in $C^1(\overline{\Omega})$ and 
	 $\hat{u}_n=\hat{u}_{v_n}=\sigma(v_n)$, $n\in{\mathbb N}$. We have
	\begin{eqnarray}\label{eq65}
		&&\left\langle A(\hat{u}_n),h\right\rangle+\int_{\Omega}\xi(z)\hat{u}^{p-1}_nhdz=\int_{\Omega}[\hat{u}_n^{-\gamma}+f(z,\hat{u}_n,Dv_n)]hdz\\
		&&\mbox{for all}\ h\in W^{1,p}(\Omega),\  n\in{\mathbb N}.\nonumber
	\end{eqnarray}
	
	Choosing $h=\hat{u}_n\in W^{1,p}(\Omega)$, we obtain
	\begin{eqnarray*}
		&&||D\hat{u}_n||^p_p+\int_{\Omega}\xi(z)\hat{u}^p_pdz\leq\int_{\Omega}c_7[\tilde{u}^{-\gamma}+1]dz\ \mbox{for some}\ c_7>0,\ \mbox{and all}\ n\in{\mathbb N}\\
		&&(\mbox{since}\ \tilde{u}\leq\hat{u}_n\leq w\ \mbox{for all}\ n\in{\mathbb N}\ \mbox{and due to hypothesis}\ H(f)(ii)),\\
		&\Rightarrow&c_1||\hat{u}_n||^p\leq c_8\ \mbox{for some}\ c_8>0\ \mbox{and all}\ n\in{\mathbb N}\ (\mbox{see Lemma \ref{lem3}}),\\
		&\Rightarrow&\{\hat{u}_n\}_{n\in{\mathbb N}}\subseteq W^{1,p}(\Omega)\ \mbox{is bounded}.
	\end{eqnarray*}
	
	Invoking Proposition 7 of Papageorgiou \& R\u{a}dulescu \cite{19}, we have
	$$||\hat{u}_n||_{\infty}\leq c_9\ \mbox{for some}\ c_9>0\ \mbox{and all}\ n\in{\mathbb N}.$$
	
	Then Theorem 2 of Lieberman \cite{15} implies that we can find $\beta\in(0,1)$ and $c_{10}>0$ such that
	\begin{equation}\label{eq66}
		\hat{u}_n\in C^{1,\beta}(\overline{\Omega})\ \mbox{and}\ ||\hat{u}_n||_{C^{1,\beta}(\overline{\Omega})}\leq c_{10}\ \mbox{for all}\ n\in{\mathbb N}.
	\end{equation}
	
	The compact embedding of $C^{1,\beta}(\overline{\Omega})$ into $C^1(\overline{\Omega})$ and (\ref{eq66})
	 imply that at least for a subsequence, we have
	\begin{equation}\label{eq67}
		\hat{u}_n\rightarrow\hat{u}\ \mbox{in}\ C^1(\overline{\Omega})\ \mbox{as}\ n\rightarrow\infty.
	\end{equation}
	
	Passing to the limit as $n\rightarrow\infty$ in (\ref{eq65}), we can infer that $\hat{u}\in S_v$.
	
	We know that $\sigma(v)\in S_v$ and so by Lemma \ref{lem8}, we can find $u_n\in S_{v_n}$ (for all $n\in{\mathbb N}$) such that
	\begin{equation}\label{eq68}
		u_n\rightarrow\sigma(v)\ \mbox{in}\ C^1(\overline{\Omega})\ \mbox{as}\ n\rightarrow+\infty.
	\end{equation}
	
	We have
	\begin{eqnarray*}
		&&\hat{u}_n\leq u_n\ \mbox{for all}\ n\in{\mathbb N},\\
		&\Rightarrow&\hat{u}\leq\sigma(v),\\
		&\Rightarrow&\sigma(v)=\hat{u}\ (\mbox{since}\ \hat{u}\in S_v).
	\end{eqnarray*}
	
	So, for the original sequence $\{\hat{u}_n=\sigma(v_n)\}_{n\in{\mathbb N}}\subseteq C^1(\overline{\Omega})$, we have
	\begin{eqnarray*}
		&&\sigma(v_n)=\hat{u}_n\rightarrow\hat{u}=\sigma(v)\ \mbox{in}\ C^1(\overline{\Omega}),\\
		&\Rightarrow&\sigma(\cdot)\ \mbox{is continuous}.
	\end{eqnarray*}
	
	Next, let $B\subseteq C^1(\overline{\Omega})$ be bounded. As before, we obtain
	\begin{eqnarray*}
		&&\sigma(B)\subseteq W^{1,p}(\Omega)\ \mbox{is bounded},\\
		&\Rightarrow&\sigma(B)\subseteq L^{\infty}(\Omega)\ \mbox{is bounded (see \cite{19})}.
	\end{eqnarray*}
	
	Then by Lieberman \cite{15} we conclude that
	$$\overline{\sigma(B)}\subseteq C^1(\overline{\Omega})\ \mbox{is compact}.$$
	
	This proves that the minimal solution map $\sigma(\cdot)$ is compact.
	The proof 
of Proposition~\ref{prop9}
is now complete.
\end{proof}

Now using Theorem \ref{th1} (the Leray-Schauder alternative principle),
 we will produce a positive smooth solution for problem (\ref{eq1}).
\begin{theorem}\label{th10}
	If hypotheses $H(\xi),\,H(f)$ hold, then problem (\ref{eq1}) admits a positive solution $u^*\in D_+$.
\end{theorem}
\begin{proof}
	We consider the minimal solution map $\sigma: C^1(\overline{\Omega})\rightarrow C^1(\overline{\Omega})$. From Proposition \ref{prop9} we know that $\sigma(\cdot)$ is compact. Let
	$$K=\{u\in C^1(\overline{\Omega}):u=t\sigma(u),0<t<1\}.$$
	
	We claim that $K\subseteq C^1(\overline{\Omega})$ is bounded. So, let $u\in K$. We have
	$$\frac{1}{t}u=\sigma(u)\ \mbox{with}\ 0<t<1.$$
	
	Then
	\begin{eqnarray}\label{eq69}
		&&\left\langle A(u),h\right\rangle+\int_{\Omega}\xi(z)u^{p-1}hdz=t^{p-1}\int_{\Omega}\left[\frac{t^{\gamma}}{u^{\gamma}}+f(z,\frac{1}{t}u,Du)\right]hdz\\
		&&\mbox{for all}\ h\in W^{1,p}(\Omega).\nonumber
	\end{eqnarray}
	
	From (\ref{eq15}) (see hypothesis $H(f)(iii)$), we have
	\begin{equation}\label{eq70}
		f(z,\frac{1}{t}u(z),Du(z))\leq\frac{1}{t^{p-1}}f(z,u(z),Du(z))\ \mbox{for almost all}\ z\in\Omega.
	\end{equation}
	
	Using (\ref{eq70}) in (\ref{eq69}) and recalling that $\tilde{u}\leq u,\ 0<t<1$, we obtain
	\begin{equation}\label{eq71}
		\left\langle A(u),h\right\rangle+\int_{\Omega}\xi(z)u^{p-1}hdz\leq\int_{\Omega}\left[\frac{1}{\tilde{u}^{\gamma}}+\hat{a}_0(z)\right]hdz
	\end{equation}
	for all $h\in W^{1,p}(\Omega)$
	and
	 some $\hat{a}_0\in L^{\infty}(\Omega)$ (see hypothesis $H(f)(i)$).
	
	In (\ref{eq71}) we choose $h=u\in W^{1,p}(\Omega)$. Then
	\begin{eqnarray*}
		&&\vartheta(u)\leq c_{11}\ \mbox{for some}\ c_{11}>0\ (\mbox{recall}\ \tilde{u}\in D_+),\\
		&&c_1||u||^p\leq c_{11}\ \mbox{for all}\ u\in K\ (\mbox{see Lemma \ref{lem3}}),\\
		&\Rightarrow&K\subseteq W^{1,p}(\Omega)\ \mbox{is bounded}.
	\end{eqnarray*}
	
	Next, as before, the nonlinear regularity theory implies that
	$$K\subseteq C^1(\overline{\Omega})\ \mbox{is bounded (in fact, relatively compact)}.$$
	
	So, we can apply Theorem \ref{th1} (the Leray-Schauder principle) and produce $u^*\in C^1(\overline{\Omega})$ such that $u^*=\sigma(u^*)$. Therefore $u^*\in D_+$ is a positive smooth solution of problem (\ref{eq1}).
	The proof 
of Theorem~\ref{th10}
is now complete.
\end{proof}

\medskip
{\bf Acknowledgements.} The authors  thank the referee for corrections and remarks.
This research was supported by the Slovenian Research Agency grants
P1-0292, J1-8131, N1-0064, N1-0083, and N1-0114.

\end{document}